\documentclass[11pt]{amsart}
\usepackage{a4wide}
\usepackage{mathrsfs}
\usepackage{amssymb, amsmath}
\usepackage{enumitem}
\newcommand{\pate}[3]{\left\langle#1| #2|#3\right\rangle}
\newcommand{\sA}[1]{\mathscr{A}_{#1}}
\newcommand{\sC}[1]{\mathscr{C}_{#1}}
\newcommand{\sF}[1]{\mathscr{F}_{#1}}
\newcommand{\NN}{\mathbf{N}}
\newcommand{\ZZ}{\mathbf{Z}}

\newcommand{\RR}{\mathbf{R}}
\newcommand{\QQ}{\mathbf{Q}}
\newcommand{\del}{\partial}

\newcommand{\fhi}{\varphi}
\newcommand{\ellal}{\ell^\infty_\mathscr{A}}
\newcommand{\ellF}{\ell^\infty_\mathscr{F}}
\newcommand{\ellalt}{\ell^\infty_\mathrm{alt}}
\newcommand{\Aut}{\mathrm{Aut}}
\newcommand{\hbc}{\mathrm{H}_\mathrm{cb}}
\newcommand{\hb}{\mathrm{H}_\mathrm{b}}
\newcommand{\PSL}{\mathbf{PSL}}
\newcommand{\SL}{\mathbf{SL}}
\newcommand{\PGL}{\mathbf{PGL}}
\newcommand{\GL}{\mathbf{GL}}
\theoremstyle{plain}
\newtheorem{thm}{Theorem}

\newtheorem{prop}[thm]{Proposition}
\newtheorem{cor}[thm]{Corollary}
\begin{document}

\title[The bounded cohomology of $\SL_2$]{The bounded cohomology of $\SL_2$\\ over local fields and $S$-integers.}

\begin{abstract}
It is proved that the continuous bounded cohomology of $\SL_2(k)$ vanishes in all positive degrees whenever $k$ is a non-Archimedean local field. This holds more generally for boundary-transitive groups of tree automorphisms and implies low degree vanishing for $\SL_2$ over $S$-integers.
\end{abstract}

\author[M. Bucher and N. Monod]{Michelle Bucher and Nicolas Monod}
\maketitle

\section{Introduction}
Bounded cohomology, despite its many applications, has only been computed in few instances and usually in low degrees. When it does not vanish, it is often enormous, notably for discrete groups with hyperbolic properties. It has been suggested (e.g.~\cite[Pr.~A]{MonodICM}) that the situation could be less exotic for Lie groups and algebraic groups; the present work submits evidence towards this hope in the case of $\SL_2$. Our methods have an arboreal component and therefore the first result reads as follows.

\begin{thm}\label{thm:main}
Let $G$ be a locally compact group acting properly on a locally finite tree $T$. If the $G$-action on the boundary at infinity $\partial T$ is transitive, then the continuous bounded cohomology $\hbc^n(G)$ vanishes for all $n>0$.
\end{thm}

Here $\hbc^n(G) = \hbc^n(G, \RR)$ refers to cohomology with the trivial coefficients $\RR$. We emphasize that the statement fails for non-trivial representations even if they are irreducible unitary representations, as can be seen for $n=2$ by cohomological induction from free lattices. The statement of the theorem was previously known only for $n=2$ (where it follows from~\cite[7.1]{Burger-Monod1}; see~\cite{Iozzi-Pagliantini-Sisto_arx} for a more precise characterisation).

\medskip
Theorem~\ref{thm:main} applies notably to Bruhat--Tits trees associated to $\SL_2$, yielding the first complete computation of the bounded cohomology of a simple algebraic group:

\begin{cor}\label{cor:SL}
Let $k$ be a non-Archimedean local field. Then the continuous bounded cohomology of $\SL_2(k)$ vanishes in all positive degrees.

The same holds for $\GL_2(k)$, $\PSL_2(k)$ and $\PGL_2(k)$.\qed
\end{cor}

This corollary stands in contrast to the Archimedean case: currently, $\hbc^n(\SL_2(\RR))$ is only known up to $n=4$. Specifically, it is one-dimensional for $n=2$ (combining~\cite{Guichardet-Wigner} with~\cite[6.1]{Burger-Monod1}), it vanishes for $n=3$ by~\cite[1.5]{Burger-MonodERN} and vanishes for $n=4$ by~\cite{Hartnick-Ott15}. (We recall that for any group, $\hbc^1$ vanishes for trivial reasons.)

\medskip

Using a result from~\cite{MonodVT}, we can deduce vanishing results for irreducible lattices in products of trees. In the following statement, the lattice $\Gamma$ is called \textbf{irreducible} if its projection to any proper sub-product of $G_i$s is dense.

\begin{cor}\label{cor:lattices}
Let $\Gamma<G_1\times \cdots \times G_\ell$ be an irreducible lattice, where each $G_i$ is a locally compact group acting properly on a locally finite tree $T_i$, transitively on $\partial T_i$.

Then the bounded cohomology $\hb^n(\Gamma)$ vanishes for all $0<n< 2\ell$.
\end{cor}

This was previously known only for $n=2$, by Corollary~26 of~\cite{Burger-Monod3}.

\medskip

Because $S$-arithmetic groups can be realized as irreducible lattices in mixed Archimedean and non-Archimedean groups, one can also establish a result for $\SL_2$ over $S$-integers.

\begin{cor}\label{cor:sadic}
Let $S$ be a finite set of prime numbers.

Then $\hb^n(\SL_2(\ZZ[S^{-1}]))$ is isomorphic to $\hbc^n(\SL_2(\RR))$ for all $n < 2|S| +2$.
\end{cor}

This was previously known only for $n=2$ (see~\cite[Cor.~24]{Burger-Monod3}). Our bound on $n$ is sharp in the special case $S=\varnothing$, since the second bounded cohomology of $\SL_2(\ZZ)$ is infinite-dimensional.

\medskip

Our proofs rely on a new complex of \emph{aligned} chains which we show to be acyclic in the context of bounded operators. Continuing into higher rank, we introduce the \emph{flatmate complex}, but currently we cannot prove the corresponding acyclicity result. The rest of our proofs however still works and this raises the prospect that the continuous bounded cohomology (with trivial coefficients) of \emph{any} $p$-adic group could vanish in all positive degrees. A very partial result in this direction is the following, which is deduced from Corollary~\ref{cor:SL} using a method from~\cite{MonodJAMS}.

\begin{cor}\label{cor:H3}
Let $k$ be a non-Archimedean local field and $d\in \NN$.

Then $\hbc^n(\GL_d(k))$ vanishes for $0<n\leq 3$.
\end{cor}

The case $n\leq 2$ was established in~\cite[6.1]{Burger-Monod1}.

\section{The complex of aligned chains}
Given a (non-emtpy) set $X$ and $n\in\NN$, we denote by $\sC n(X, \ZZ)$ the group of $\ZZ$-valued \textbf{alternating $n$-chains} on $X$, which is the quotient of the free $\ZZ$-module on $X^{n+1}$ by the equivalence relation identifying a $(n+1)$-tuple $(x_0, \dots, x_n)$ with $\mathrm{sgn}(\sigma) (x_{\sigma(0)},\dots,x_{\sigma(n)})$, where $\mathrm{sgn}(\sigma)$ is the signature of an arbitrary permutation $\sigma$ of $\{0, \dots, n\}$. We shall still denote simply by $(x_0, \dots, x_n)$ the equivalence class of the basis element determined by this $(n+1)$-tuple. We obtain the standard resolution
\begin{equation}\label{eq:stdres}
0 \longleftarrow \ZZ \longleftarrow \sC 0(X, \ZZ)  \longleftarrow \sC 1(X, \ZZ)  \longleftarrow \sC 2(X, \ZZ)  \longleftarrow \cdots
\end{equation}
with boundary map $\del\colon \sC n(X, \ZZ) \to \sC {n-1}(X, \ZZ)$ determined by $\del = \sum_{j=0}^n (-1)^j \del_j$, where $\del_j$ omits the $j$-th variable, and augmentation  $\sC 0(X, \ZZ) \to \ZZ$ given by the summation. This is all well-defined under the identification given by the ``alternation'' skew-symmetry.

It is elementary and well-known that~\eqref{eq:stdres} is indeed an exact sequence. All this can be done without change for the vector spaces $\sC *(X, \RR)$ of $\RR$-valued chains.

\bigskip

Let now $T$ be a simplicial tree; we abusively also denote by $T$ the set of its vertices. We introduce the subcomplex of alternating \textbf{aligned chains} $\sA *(T, \ZZ)$ in $\sC *(T, \ZZ)$ by defining $\sA n(T, \ZZ)$ to be spanned by those $(n+1)$-tuples that are contained in some geodesic segment in $T$. Thus, unless $T$ is a linear tree, $\sA n(T, \ZZ)$ is a proper subgroup of  $\sC n(T, \ZZ)$ as soon as $n\geq 2$. Again, the same definitions are introduced over $\RR$.

In order to use these chain complexes for bounded cohomology, we endow $\sC n(T, \RR)$ and $\sA n(T, \RR)$ with the quotient norm induced by the $\ell^1$-norm on $\ell^1(T^{n+1})$. The boundary $\del$ is a bounded operator for these norms. Thus, taking the dual of the normed spaces, we obtain cochain complexes
$$0 \longrightarrow \RR \longrightarrow \ellalt(T)  \longrightarrow \ellalt(T^2)  \longrightarrow \ellalt(T^3)  \longrightarrow \cdots$$
$$0 \longrightarrow \RR \longrightarrow \ellal(T)  \longrightarrow \ellal(T^2)  \longrightarrow \ellal(T^3)  \longrightarrow \cdots$$
wherein $\ellalt(T^{n+1})$ denotes the Banach space of bounded alternating functions on $T^{n+1}$ and $\ellal(T^{n+1})$ the subspace of those that are supported on aligned tuples. More generally, for any Banach space $V$ we can identify the Banach space $\ellal(T^{n+1}, V)$ of $V$-valued bounded functions on aligned tuples with the space of bounded operators from $\sA n(T, \RR)$ to $V$.

\medskip
It can be shown that the complex $\sA *(T, \ZZ)$ is a direct summand of  $\sC *(T, \ZZ)$. The point of the following result is that there exists such a splitting that is moreover realized by bounded operators.

\begin{thm}\label{thm:fhi}
There exists a natural chain map $\fhi_*\colon \sC *(T, \RR)\to \sA *(T, \RR)$ such that each $\fhi_n$ is a bounded projection onto  $\sA n(T, \RR)$. Moreover, $\fhi_n$ sends $\sC n(T, \ZZ)$ to $\sA n(T, \ZZ)$.  
\end{thm}

The naturality of $\fhi_*$ makes it commute with the automorphisms of $T$, and hence Theorem~\ref{thm:fhi} implies the following since~\eqref{eq:stdres} is exact.

\begin{cor}\label{cor:res}
For any Banach space $V$, we have a natural resolution
$$0 \longrightarrow V \longrightarrow \ellal(T,V)  \longrightarrow \ellal(T^2,V)  \longrightarrow \ellal(T^3,V)  \longrightarrow \cdots$$
of $V$. If $V$ is endowed with an isometric linear representation of the automorphism group $\Aut(T)$ of $T$, then the above is a resolution of $V$ by isometric Banach $\Aut(T)$-modules.\qed
\end{cor}

\begin{proof}[Proof of Theorem~\ref{thm:fhi}]
Given a $(n+1)$-tuple $x=(x_0, \dots, x_n)$ and $1\leq i,j\leq n$, we denote by $p_{i,j}^x\colon T\to T$ the map that sends a vertex to its nearest point projection on the geodesic segment $[x_i, x_j]$. This depends of course on the ordering of the tuple; however, summing over all pairs $i\leq j$, we obtain a well-defined map $\fhi_n\colon \sC n(T, \ZZ) \to \sA n(T, \ZZ)$ determined on tuples by
\begin{equation}\label{eq:fhi}
\fhi_n (x) = \sum_{i\leq j} p_{i,j}^x(x), \kern5mm\text{where } p_{i,j}^x(x) = (p_{i,j}^x(x_0), \dots, p_{i,j}^x(x_n)).
\end{equation}
Moreover, $\fhi_n$ yields a bounded operator $\sC n(T, \RR)\to \sA n(T, \RR)$.

In order to study the sum~\eqref{eq:fhi} defining $\fhi_n (x)$, we consider the finite tree $T(x)$ obtained as the convex hull in $T$ of the coordinates of $x$. In particular, all leaves of $T(x)$ are coordinates of $x$. We first observe that $p_{i,j}^x(x)$ vanishes in $\sA n(T, \ZZ)$ as soon as either $x_i$ or $x_j$ is not a leaf of $T(x)$, since in that case $p_{i,j}^x(x)$ contains repetitions. In particular, if $x$ is already aligned, then the only possibly non-zero term is given by the indices $i\leq j$ such that $\{x_i, x_j\}$ are the extremal points of the interval $T(X)$. This already shows that $\fhi_n$ is a projection onto $\sA n(T, \ZZ)$.

\medskip
The main point is to show that $\fhi_*$ is a chain map, i.e.\ that $\del \fhi_n = \fhi_{n-1}\del$ holds. Since $\fhi_*$ is the identity on $\sC 0$ and $\sC 1$, we assume $n\geq 2$. Choose some $x\in X^{n+1}$.

Most terms in the sum~\eqref{eq:fhi} vanish; more precisely, we observe that any pair $i\leq j$ with $p_{i,j}^x(x) \neq 0$ in $\sA n(T, \ZZ)$ has the following property: $x_i$ and $x_j$ are (distinct) leaves of $T(x)$ and all $\bar x_k:= p_{i,j}^x(x_k)$ are distinct ($k=0, \dots, n$). Indeed, in all other cases $p_{i,j}^x(x)$ would contain repetitions. In other words, upon renumbering the indices so that the $\bar x_k$ are linearly ordered on $[x_i, x_j]$, the tree $T(x)$ has the following configuration
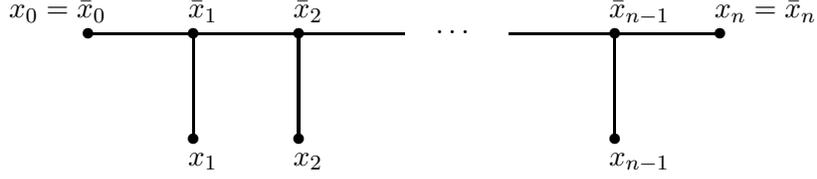
\begin{center}
\setlength{\unitlength}{0.7cm}
\thicklines
\begin{picture}(16,4)
\put(2,3){\line(1,0){6}}
\put(10,3){\line(1,0){4}}
\put(2,3){\circle*{0.2}}\put(0.5,3.3){$x_0=\bar x_0$}
\put(4,3){\circle*{0.2}}\put(3.9,3.3){$\bar x_1$}
\put(4,1){\circle*{0.2}}\put(3.9,0.5){$x_1$}
\put(4,3) {\line(0,-1){2}}
\put(6,3){\circle*{0.2}}\put(5.9,3.3){$\bar x_2$}
\put(6,1){\circle*{0.2}}\put(5.9,0.5){$x_2$}
\put(6,3) {\line(0,-1){2}}
\put(8.6,3){$\ldots$}
\put(12,3){\circle*{0.2}}\put(11.9,3.3){$\bar x_{n-1}$}
\put(12,1){\circle*{0.2}}\put(11.9,0.5){$x_{n-1}$}
\put(12,3) {\line(0,-1){2}}
\put(14,3){\circle*{0.2}}\put(13.9,3.3){$x_n=\bar x_n$}
\end{picture}
\end{center}
where all horizontal segments have non-zero length; let us call this a \emph{standard configuration} for now. In conclusion, there are at most four non-zero terms in the sum~\eqref{eq:fhi}, corresponding to $i=0,1$ and $j=n-1, n$ once $x$ is in a standard configuration.

At this point we introduce a notation for arbitrary tuples which does not use the tree structure of $T$: given a $(n-1)$-tuple $z=(z_0, \dots z_{n-2})$ and two pairs $u=(u', u'')$ and $v=(v', v'')$, we define a sum of $(n+1)$-tuples, and hence also an element of $\sC n(T, \ZZ)$, by the formula
$$\pate uzv := - (u', z, v') + (u', z, v'') - (u'', z, v'') + (u'', z, v').$$
Notice that the value of $\pate uzv$ in $\sC n(T, \ZZ)$ depends only on $u,v$ as elements of $\sC 1(T, \ZZ)$ and on $z$ as an element of $\sC {n-2}(T, \ZZ)$. We have the cocycle relation
\begin{equation}\label{eq:coqpate}
\pate{(u',u'')}zv =  \pate{(u',u''')}zv + \pate{(u''',u'')}zv
\end{equation}
and the corresponding relation with $u$ and $v$ interchanged. Furthermore, we have
\begin{equation}\label{eq:delpate}
\del_j \pate uzv  = \begin{cases}
0 & \text{if $j= 0,n$,}\\
\pate{u}{\del_{j-1} z}{v} & \text{otherwise.}
\end{cases}
\end{equation}

\smallskip
We now return to the four possibly non-vanishing terms of $\fhi_n(x)$ for a standard configuration of $x$. Using skew-symmetry, we can rewrite them as
\begin{equation}\label{eq:fhipate}
\fhi_n(x) =  \pate {(x_0, x_1)}{(\bar x_1, \dots, \bar x_{n-1})}{(x_{n-1}, x_n)}.
\end{equation}
This rewriting uses the fact that the projection of $x_0$ onto $[x_1, x_n]$ and onto $[x_1, x_{n-1}]$ is $\bar x_1$, and similarly that the two relevant projections of $x_n$ are $\bar x_{n-1}$. Notice that this equation does indeed also hold in the case $n=2$, even though only three terms are possibly non-zero.

In order to conclude our proof of $\del \fhi_n = \fhi_{n-1}\del$, we first assume $n\geq 3$. Notice that for any $j$ the $(n-1)$-tuple $\del_j x$ still is in a standard configuration. Therefore, if $2\leq j \leq n-2$, then~\eqref{eq:delpate} and~\eqref{eq:fhipate} imply $\del_j \fhi_n (x) = \fhi_{n-1}\del_j (x)$. Therefore, since $\del_0 \fhi_n (x)$ and $\del_n \fhi_n (x)$ vanish, it suffices to verify that  $-\del_1 \fhi_n (x) = \fhi_{n-1}\del_0 (x) - \fhi_{n-1}\del_1 (x)$ and that $\del_{n-1} \fhi_n (x) = \fhi_{n-1}\del_{n-1} (x) - \fhi_1\del_n (x)$. We check the first since the other follows by symmetry. The relation~\eqref{eq:coqpate} implies with~\eqref{eq:fhipate}
$$\fhi_{n-1}\del_0 (x) - \fhi_{n-1}\del_1 (x) = - \pate{(x_0, x_1)}{(\bar x_2, \dots, \bar x_{n-1})}{(x_{n-1}, x_n)}$$
which is indeed $-\del_1 \fhi_n (x)$ by~\eqref{eq:delpate} and~\eqref{eq:fhipate}.

\smallskip
Finally, we examine the special case $n=2$; thus we consider $x=(x_0, x_1, x_2)$ in a standard configuration. Then
\begin{align*}
\fhi_2(x) &= (\bar x_1, x_1, x_2) + (x_0, \bar x_1, x_2) + (x_0, x_1, \bar x_1)\\
&= x + \del (x_0, x_1, x_2, \bar x_1)
\end{align*}
(by skew-symmetry), and therefore $\del\fhi_2(x) = \del x$, which is $\fhi_1 \del x$ as required.
\end{proof}

\section{Proof of Theorem~\ref{thm:main}}
The image of $G$ in $\Aut(T)$ is closed since the action is proper. We can assume that it is non-compact and that $T$ is a thick tree, since otherwise $G$ would be amenable which implies the vanishing of its bounded cohomology. Thus, the transitivity on the boundary $\partial T$ implies that $G$ is actually $2$-transitive on $\partial T$, see e.g. Lemma~3.1.1 in~\cite{Burger-Mozes1}. Thus $G$ is transitive on the set of geodesic lines in $T$, and in fact even \emph{strongly transitive} in the sense that it acts transitively on the set of pairs consisting of a geodesic line and a geometric (i.e.\ unoriented) edge in that line, see e.g. Corollary~3.6 in~\cite{Caprace-Ciobotaru}.

Recall that there is a homomorphism $G\to\ZZ/2\ZZ$ given by the parity of the displacement length of some, or equivalently any, vertex of $T$; in particular, the kernel $G^+< G$ of this homomorphism acts without inversions. It suffices to prove the theorem for $G^+$ instead of $G$ since the restriction from $G$ to $G^+$ in bounded cohomology is injective~\cite[8.8.5]{Monod}. On the other hand, the criterion~(2) of Lemma~3.1.1 in~\cite{Burger-Mozes1} shows that $G^+$ still is $2$-transitive on $\partial T$ because any point-stabiliser in $G$ is contained in $G^+$. We can therefore assume $G=G^+$.

Since the $G$-action on the various sets of aligned tuples is proper, the $G$-module $\ellal(T^{n+1}, \RR)$ is relatively injective in the sense of bounded cohomology for all $n\geq 0$, see e.g.~\cite[4.5.2]{Monod}. Therefore, Corollary~\ref{cor:res} implies that $\hbc^*(G)$ is realized by the complex of $G$-invariants
$$0 \longrightarrow \ellal(T,\RR)^G  \longrightarrow \ellal(T^2,\RR)^G  \longrightarrow \ellal(T^3,\RR)^G  \longrightarrow \cdots$$
Choose a geodesic line $L$ in $T$ and denote by $H<G$ the (setwise) stabiliser of $L$ in $G$; notice that $H$ is amenable. In particular, it is trivial for bounded cohomology. Therefore, the following proposition concludes the proof of Theorem~\ref{thm:main} (noting that the restriction is a cochain map).

\begin{prop}\label{prop:res}
The restriction to $L$ determines an isomorphism
$$\ellal(T^{n+1},\RR)^G \xrightarrow{\ \cong\ }  \ellalt(L^{n+1},\RR)^H$$
for all $n\geq 0$.
\end{prop}

\begin{proof}[Proof of Proposition~\ref{prop:res}]
The map is injective because every aligned tuple lies on some geodesic line and $G$ can send this line to $L$ since it is transitive on lines.

For surjectivity, we fix a function $f$ in $\ellalt(L^{n+1},\RR)^H$ and proceed to extend it to a function $\widetilde f$ in $\ellal(T^{n+1},\RR)^G$ as follows. Choose an aligned tuple $x=(x_0, \dots, x_n)$ in $T$; we can assume that it lies within the segment $[x_0, x_n]$. By transitivity on lines, there is $g\in G$ such that $gx$ lies on $L$. We claim that if $g'$ is any other such element, then the values $f(gx)$ and $f(g'x)$ coincide. This claim allows us to define $\widetilde f(x) = f(gx)$ in a well-posed and $G$-invariant manner.

To prove the claim, it suffices to find $h\in H$ such that $h gx = g'x$. Since $x$ lies in $[x_0, x_n]$, it is enough to ensure $hg x_0 = g' x_0$ and $hg x_n = g' x_n$. Let $y$ be the first vertex after $x_0$ in the segment $[x_0, x_n]$ (noting that the case $x_0= x_n$ would be trivial). The strong transitivity of $G$ on $T$ implies that $H$ is transitive on the geometric edges of $L$. Therefore, there is $h\in H$ with $h\{g x_0, g y\} = \{g' x_0, g' y\}$. If we now establish $h g x_0 = g' x_0$, then $hg x_n = g' x_n$ follows and the claim holds. It remains thus to exclude $h g x_0 = g' y$. The latter inequality is impossible because $x_0$ is adjacent to $y$ and we reduced to the case $G=G^+$.
\end{proof}

\subsection*{Scholium: the flatmate complex}
More generally, let $X$ be a building. We still abusively denote by $X$ the set of its vertices. We define the subcomplex of alternating \textbf{flatmate chains} $\sF *(X, \ZZ)$ in $\sC *(X, \ZZ)$ by defining $\sF n(X, \ZZ)$ to be spanned by those $(n+1)$-tuples that are contained in some apartment. We define the corresponding complex $\sF *(X, \RR)$ of normed vector spaces.

\smallskip
\itshape
We conjecture that $\sF *(X, \RR)$ admits a bounded contracting homotopy when $X$ is an affine building.\upshape

\medskip

If this is the case, then the continuous bounded cohomology of $\Aut(X)$ with coefficient in a Banach module $V$ is realised on the complex
$$0 \longrightarrow \ellF(X,V)^{\Aut(X)}  \longrightarrow \ellF(X^2,V)^{\Aut(X)}  \longrightarrow \ellF(X^3,V)^{\Aut(X)}  \longrightarrow \cdots$$
where $\ellF$ denotes the bounded alternating function supported on tuples that lie in some apartments. Indeed, the conjecture provides a resolution as in Corollary~\ref{cor:res} and this resolution is relatively injective because the action of $\Aut(X)$ on any tuples is proper.

At this point, our proof of Theorem~\ref{thm:main} can be translated faithfully to the present context provided $\Aut(X)$ is \emph{strongly transitive}, i.e.\ acts transitively on the set of apartments with a distinguiched chamber. In particular, this applies to Bruhat--Tits buildings.

In conclusion, the above conjecture implies the vanishing of the bounded cohomology of any semi-simple group over a non-Archimedean local field $k$. This, in turn generalizes to algebraic groups over $k$ since we do not change the bounded cohomology when quotienting out the amenable radical.

\section{Remaining proofs}
We start by recalling the particular case of products in the Hochschild--Serre spectral sequence for the continuous bounded cohomology, following~\cite[\S12]{Monod}. Let $G=G_1\times G_2$ be a product of locally compact groups.

Although there is indeed a natural spectral sequence $E_r^{p,q}$ abbuting to $\hbc^*(G)$ (viewed as abstract vector space), there is in general no usable description of the entire second tableau $E_2^{p,q}$; this difficulty occurs to some extent also for (usual) continuous cohomology. Nonetheless, if for some $q$ the semi-normed space $\hbc^q(G_2)$ is Hausdorff, then $E_2^{p,q}$ is isomorphic to $\hbc^p(G_1, \hbc^q(G_2))$ for this $q$ and all $p$. See~12.2.2(ii) in~\cite{Monod}.

\begin{proof}[Proof of Corollary~\ref{cor:lattices}]
According to Corollary~1.12 in~\cite{MonodVT}, there is an isomorphism
$$\hbc^n(G) \longrightarrow \hb^n(\Gamma)$$
for all $n<2\ell$, where $G=G_1\times \cdots \times G_\ell$. On the other hand, the above form of the Hochschild--Serre spectral sequence can be applied by induction on $\ell$ since $\hbc^q(G_i)$ is Hausdorff for all $q$ and all $i$ thanks to Theorem~\ref{thm:main}. (The degree zero bounded cohomology is always one-dimensional and Hausdorff.) We deduce that $\hbc^n(G)$ vanishes for all $n>0$.
\end{proof}

\begin{proof}[Proof of Corollary~\ref{cor:sadic}]
The group $\SL_2(\ZZ[S^{-1}])$ is an irreducible lattice in the product
$$\SL_2(\RR) \times \prod_{p\in S}  \SL_2(\QQ_p).$$
Again, we can apply  Hochschild--Serre since the bounded cohomology of $\prod_{p\in S}  \SL_2(\QQ_p)$ is Hausdorff and we conclude as in Corollary~\ref{cor:lattices} by applying Corollary~1.12 in~\cite{MonodVT} with now $\ell = 1+ |S|$.
\end{proof}

\begin{proof}[Proof of Corollary~\ref{cor:H3}]
Writing $G_d=\GL_d(k)$, we shall prove by induction on $d$ that $\hbc^3(G_d)$ vanishes, recalling that $\hbc^2(G_d)$ vanishes by~\cite[6.1]{Burger-Monod1} and that $\hbc^1$ always vanishes. Since the case $d<2$ is trivial, the starting point is $d=2$, which is taken care of by Corollary~\ref{cor:SL}. For $d=3$, we apply Proposition~3.4 in~\cite{MonodJAMS}, which states that the restriction map
$$\hbc^3(G_3) \longrightarrow \hbc^3(G_2)$$
is injective. For $d\geq 4$, the restriction $\hbc^3(G_d) \to \hbc^3(G_{d-1})$ is an isomorphism by Theorem~1.1 in~\cite{MonodJAMS}.
\end{proof}

\medskip\noindent
\textbf{Acknowledgements.} It is a pleasure to thank Pierre-Emmanuel Caprace for his remarks on a preliminary version of this article.


\bibliographystyle{../BIB/amsplain}
\bibliography{../BIB/ma_bib}

\end{document}